\documentclass[runningheads]{llncs}

\usepackage[T1]{fontenc}

\usepackage{graphicx}

\usepackage{hyperref}
\usepackage{color}

\usepackage{amsmath,amssymb,amsfonts}
\usepackage[noend]{algorithmic}
\usepackage{graphicx}
\usepackage{textcomp}
\usepackage{xcolor}

\DeclareMathOperator*{\argmin}{argmin}
\DeclareMathOperator*{\argmax}{argmax}

\usepackage{float}
\newcommand{\Matlab}{\textsc{Matlab}}
\newcommand{\mlin}[1]{\mbox{\tt{#1}}}

\newcommand{\Stiefel}{\mathrm{St}}
\DeclareMathOperator{\TTrank}{rank_\mathrm{TT}}
\DeclareMathOperator{\rank}{rank}
\usepackage{setspace}

\usepackage{cleveref}
\usepackage{algorithm}

\begin{document}
\title{An Approximate Projection onto the Tangent Cone to the Variety of Third-Order Tensors of Bounded Tensor-Train Rank}
\titlerunning{Approximate projection tangent cone third-order TTDs}
%
\author{Charlotte Vermeylen\inst{1}\orcidID{0000-0002-3615-6571}  \and
Guillaume Olikier\inst{2}\orcidID{0000-0002-2767-0480} \and
Marc Van Barel\inst{1}\orcidID{0000-0002-7372-382X}}
\authorrunning{C. Vermeylen et al.}
%
\institute{Department of Computer Science, KU Leuven, Heverlee, Belgium \email{charlotte.vermeylen@kuleuven.be} \and
ICTEAM Institute, UCLouvain, Louvain-la-Neuve, Belgium}
\maketitle              
\begin{abstract}
An approximate projection onto the tangent cone to the variety of third-order tensors of bounded tensor-train rank is proposed and proven to satisfy a better angle condition than the one proposed by Kutschan (2019). Such an approximate projection enables, e.g., to compute gradient-related directions in the tangent cone, as required by algorithms aiming at minimizing a continuously differentiable function on the variety, a problem appearing notably in tensor completion. A numerical experiment is presented which indicates that, in practice, the angle condition satisfied by the proposed approximate projection is better than both the one satisfied by the approximate projection introduced by Kutschan and the proven theoretical bound.

\end{abstract}
\section{Introduction}
Tangent cones play an important role in constrained optimization to describe admissible search directions and to formulate optimality conditions \cite[Chap. 6]{RockafellarWets}. In this paper, we focus on the set
\begin{equation}
\label{eq:LowTTrankVariety}
\mathbb{R}_{\le (k_1, k_2)}^{n_1 \times n_2 \times n_3} := \{X \in \mathbb{R}^{n_1 \times n_2 \times n_3} \mid \TTrank(X) \le (k_1, k_2)\},
\end{equation}
where $\TTrank(X)$ denotes the tensor-train rank of $X$ (see \Cref{subsec:TTDs}), which is a real algebraic variety \cite{kutschan2018}, and, given $X \in \mathbb{R}_{\le (k_1, k_2)}^{n_1 \times n_2 \times n_3}$, we propose an \emph{approximate projection} onto the tangent cone $T_X \mathbb{R}_{\le (k_1,k_2)}^{n_1 \times n_2 \times n_3}$, i.e., a set-valued mapping $\tilde{\mathcal{P}}_{T_X \mathbb{R}_{\le (k_1, k_2)}^{n_1 \times n_2 \times n_3}} : \mathbb{R}^{n_1 \times n_2 \times n_3} \multimap T_X \mathbb{R}_{\le (k_1, k_2)}^{n_1 \times n_2 \times n_3}$ such that there exists $\omega \in (0, 1]$ such that, for all $Y \in \mathbb{R}^{n_1 \times n_2 \times n_3}$ and all $\tilde{Y} \in \tilde{\mathcal{P}}_{T_X \mathbb{R}_{\le (k_1, k_2)}^{n_1 \times n_2 \times n_3}} Y$,
\begin{equation}
\label{eq:ApproximateProjectionAngleCondition}
\langle Y, \tilde{Y} \rangle \ge \omega \|\mathcal{P}_{T_X \mathbb{R}_{\le (k_1, k_2)}^{n_1 \times n_2 \times n_3}}Y\| \|\tilde{Y}\|,
\end{equation}
where $\langle \cdot, \cdot\rangle$ is the inner product on $\mathbb{R}^{n_1 \times n_2 \times n_3}$ given in \cite[Example~4.149]{Hackbusch}, \mbox{$\|\cdot\|$} is the induced norm, and the set
\begin{equation}
\label{eq:min_proj}
\mathcal{P}_{T_X \mathbb{R}_{\le (k_1, k_2)}^{n_1 \times n_2 \times n_3}} Y := \argmin_{Z \in T_X \mathbb{R}_{\le (k_1, k_2)}^{n_1 \times n_2 \times n_3}} \|Z-Y\|^2
\end{equation}
is the projection of $Y$ onto $T_X \mathbb{R}_{\le (k_1, k_2)}^{n_1 \times n_2 \times n_3}$.
By \cite[Definition~2.5]{schneider2015convergence}, inequality~\eqref{eq:ApproximateProjectionAngleCondition} is called an \emph{angle condition}; it is well defined since, as $T_X \mathbb{R}_{\le (k_1, k_2)}^{n_1 \times n_2 \times n_3}$ is a closed cone, all elements of $\mathcal{P}_{T_X \mathbb{R}_{\le (k_1, k_2)}^{n_1 \times n_2 \times n_3}}Y$ have the same norm (see \Cref{sec:ProposedApproximateProjection}).
Such an approximate projection enables, e.g., to compute a gradient-related direction in $T_X \mathbb{R}_{\le (k_1, k_2)}^{n_1 \times n_2 \times n_3}$, as required in the second step of \cite[Algorithm~1]{schneider2015convergence} if the latter is used to minimize a continuously differentiable function $f : \mathbb{R}^{n_1 \times n_2 \times n_3} \to \mathbb{R}$ on $\mathbb{R}_{\le (k_1, k_2)}^{n_1 \times n_2 \times n_3}$, a problem appearing notably in tensor completion; see \cite{Steinl_high_dim_TT_compl_2016} and the references therein. 

An approximate projection onto $T_X \mathbb{R}_{\le (k_1, k_2)}^{n_1 \times n_2 \times n_3}$ satisfying the angle condition \eqref{eq:ApproximateProjectionAngleCondition} with $\omega = \frac{1}{6\sqrt{n_1 n_2 n_3}}$ was proposed in \cite[\S 5.4]{kutschan2019}. If $X$ is a singular point of the variety, i.e., $(r_1, r_2) := \TTrank(X) \ne (k_1, k_2)$, the approximate projection proposed in this paper ensures (see \Cref{thm:approx_proj_TC_TT})
\begin{equation}
\label{eq:OurOmega}
\omega = \sqrt{\max\left\{\frac{k_1-r_1}{n_1-r_1}, \frac{k_2-r_2}{n_3-r_2}\right\}},
\end{equation}
which is better, and can be computed via SVDs (see \Cref{alg:HOOI_proj}). We point out that no general formula to project onto the closed cone $T_X \mathbb{R}_{\le (k_1,k_2)}^{n_1 \times n_2 \times n_3}$, which is neither linear nor convex (see \Cref{subsec:TangentConeLowRankVariety}), is known in the literature.

This paper is organized as follows. Preliminaries are introduced in \Cref{sec:Preliminaries}. Then, in \Cref{sec:ProposedApproximateProjection}, we introduce the proposed approximate projection and prove that it satisfies \eqref{eq:ApproximateProjectionAngleCondition} with $\omega$ as in \eqref{eq:OurOmega} (\Cref{thm:approx_proj_TC_TT}). Finally, in \Cref{sec:NumericalExperiment}, we present a numerical experiment where the proposed approximate projection preserves the direction better than the one from \cite[\S 5.4]{kutschan2019}.

\section{Preliminaries}
\label{sec:Preliminaries}
In this section, we introduce the preliminaries needed for \Cref{sec:ProposedApproximateProjection}. In \Cref{subsec:OrthogonalProjections}, we recall basic facts about orthogonal projections. Then, in \Cref{subsec:TTDs}, we review the tensor-train decomposition. Finally, in \Cref{subsec:TangentConeLowRankVariety}, we review the description of the tangent to $\mathbb{R}_{\le (k_1, k_2)}^{n_1 \times n_2 \times n_3}$ given in \cite[Theorem~2.6]{kutschan2018}.

\subsection{Orthogonal Projections}
\label{subsec:OrthogonalProjections}
Given $n, p \in \mathbb{N}$ with $n \ge p$, we let $\Stiefel(p, n) := \{U \in \mathbb{R}^{n \times p} \mid U^\top U = I_p\}$ denote the Stiefel manifold. For every $U \in \Stiefel(p, n)$, we let $P_U := UU^\top$ and $P_U^\perp := I_n-P_U$ denote the orthogonal projections onto the range of $U$ and its orthogonal complement, respectively.
The proof of \Cref{thm:approx_proj_TC_TT} relies on the following basic result.

\begin{lemma}
Let $A \in \mathbb{R}^{n \times m}$ have rank $r$. If $\hat{A} = \hat{U} \hat{S} \hat{V}^\top$ is a truncated SVD of rank $s$ of $A$, with $s<r$, then, for all $U \in \Stiefel(s, n)$ and all $V \in \Stiefel(s, m)$,
\begin{align}
\label{eq:ineqs_svd_trunc1}
\|P_{\hat{U}} A\| \ge \|P_U A\|,&&
\|P_{\hat{U}} A\|^2 \ge \frac{s}{r} \|A\|^2,\\ \label{eq:ineqs_svd_trunc2}
\|A P_{\hat{V}}\| \ge \|A P_V\|,&&
\|A P_{\hat{V}}\|^2 \ge \frac{s}{r} \|A\|^2.
\end{align}
\end{lemma}

\begin{proof}
By the Eckart--Young theorem, $\hat{A}$ is a projection of $A$ onto $$\mathbb{R}_{\le s}^{n \times m} := \{X \in \mathbb{R}^{n \times m} \mid \rank(X) \le s\}.$$ Thus, since $\mathbb{R}_{\le s}^{n \times m}$ is a closed cone, the same conditions as in \eqref{eq:proj_max_reform} hold. Moreover, since $\hat{S} \hat{V}^\top = \hat{U}^\top A$ and thus $\hat{A} = \hat{U} \hat{U}^\top A = P_{\hat{U}} A$, it holds that
\begin{equation*}
\lVert P_{\hat{U}} A \rVert^2 = \max\left\{\lVert A_1 \rVert^2 \mid A_1 \in \mathbb{R}_{\le s}^{n \times m},\, \langle A_1, A\rangle = \lVert A_1 \rVert^2\right\}.
\end{equation*}
Furthermore, for all $U \in \Stiefel(s, n)$,
$\left\langle P_{{U}}A, A \right\rangle = \left\langle P_{{U}} A, P_{{U}} A + P_{{U}}^\perp A \right\rangle = \left\lVert P_{{U}} A \right\rVert^2.$
Hence,
\begin{equation*}
\{ P_{U} A \mid U \in \Stiefel(s, n) \} \subseteq \{A_1 \in \mathbb{R}_{\le s}^{n \times m} \mid \langle A_1, A \rangle = \| A_1 \|^2\}.
\end{equation*}
Thus,
$\lVert P_{\hat{U}} A \rVert^2 = \max_{U \in \Stiefel(s, n)} \lVert P_{U} A \rVert^2.$
The left inequality in \eqref{eq:ineqs_svd_trunc1} follows, and the one in \eqref{eq:ineqs_svd_trunc2} can be obtained similarly.

By orthogonal invariance of the Frobenius norm and by definition of $\hat{A}$,
\begin{align*}
\lVert A \rVert^2
= \sum_{i=1}^r \sigma_i^2,&&
\lVert \hat{A} \rVert^2
= \sum_{i=1}^s \sigma_i^2,
\end{align*}
where $\sigma_1, \dots, \sigma_r$ are the singular values of $A$ in decreasing order.
Moreover, either $\sigma_s^2 \ge \frac{1}{r} \sum_{i=1}^r \sigma_i^2$ or $\sigma_s^2 < \frac{1}{r} \sum_{i=1}^r \sigma_i^2$.
In the first case, we have
\begin{equation*}
\lVert \hat{A} \rVert^2
= \sum_{i=1}^s \sigma_i^2
\ge s \sigma_s^2
\ge s \frac{\sum_{i=1}^r \sigma_i^2}{r}
= \frac{s}{r} \lVert A \rVert^2.
\end{equation*}
In the second case, we have
\begin{equation*}
\lVert \hat{A} \rVert^2
= \sum_{i=1}^r \sigma_i^2 - \sum_{i=s+1}^r \sigma_i^2
\ge \lVert A \rVert^2 - (r-s) \sigma_s^2
> \lVert A \rVert^2 - (r-s) \frac{\sum_{i=1}^r \sigma_i^2}{r}
= \frac{s}{r} \lVert A \rVert^2.
\end{equation*}
Thus, in both cases, the second inequality in \eqref{eq:ineqs_svd_trunc1} holds. The second inequality in \eqref{eq:ineqs_svd_trunc2} can be obtained in a similar way.
\qed
\end{proof}

\subsection{The Tensor-Train Decomposition}
\label{subsec:TTDs}
In this section, we review basic facts about the tensor-train decomposition (TTD) that are used in \Cref{sec:ProposedApproximateProjection}; we refer to the original paper \cite{Oseledets2011} and the subsequent works \cite{LR_tensor_methods_Steinlechner_2014,Steinl_high_dim_TT_compl_2016,Steinlechner_thesis2016} for more details.

A \emph{tensor-train decomposition} of $X \in \mathbb{R}^{n_1 \times n_2 \times n_3}$ is a factorization
\begin{equation}
\label{eq:TTD}
X = X_1 \cdot X_2 \cdot X_3,
\end{equation}
where $X_1 \in \mathbb{R}^{n_1 \times r_1}$, $X_2 \in \mathbb{R}^{r_1 \times n_2 \times r_2}$, $X_3 \in \mathbb{R}^{r_2 \times n_3}$, and `$\cdot$' denotes the contraction between a matrix and a tensor. The minimal $(r_1, r_2)$ for which a TTD of $X$ exists is called the \emph{TT-rank} of $X$ and is denoted by $\TTrank(X)$. By \cite[Lemma~4]{Holtz_manifolds_2012}, the set
\begin{equation}
\label{eq:LowTTrankManifold}
\mathbb{R}_{(k_1, k_2)}^{n_1 \times n_2 \times n_3} := \{X \in \mathbb{R}^{n_1 \times n_2 \times n_3} \mid \TTrank(X) = (k_1, k_2)\}
\end{equation}
is a smooth embedded submanifold of $\mathbb{R}^{n_1 \times n_2 \times n_3}$.

Let $X^\mathrm{L} := [X]^{n_1 \times n_2 n_3} := \mathrm{reshape}(X, n_1 \times n_2 n_3)$ and $X^\mathrm{R} = [X]^{n_1 n_2 \times n_3} := \mathrm{reshape}(X, n_1 n_2 \times n_3)$ denote respectively the left and right unfoldings of $X$. Then, $\TTrank(X) = (\rank(X^\mathrm{L}), \rank(X^\mathrm{R}))$ and the minimal rank decomposition can be obtained by computing two successive SVDs of unfoldings; see \cite[Algorithm~1]{Oseledets2011}. The contraction interacts with the unfoldings according to the following rules:
\begin{align*} 
X_1 \cdot X_2 \cdot X_3 = \left[X_1 (X_2 \cdot X_3)^\mathrm{L} \right]^{n_1 \times n_2 \times n_3},&&
X_2 \cdot X_3 = \left[X_2^\mathrm{R} X_3\right]^{r_1 \times n_2 \times n_3}.
\end{align*}

For every $i \in \{1, 2, 3\}$, if $U_i \in \Stiefel(r_i, n_i)$, then the \mbox{mode-$i$} vectors of $[P_{U_1} X^\mathrm{L}$ $ (P_{U_3} \otimes P_{U_2})]^{n_1 \times n_2 \times n_3}$ are the orthogonal projections onto the range of $U_i$ of those of $X$. A similar property holds for $X^\mathrm{R}$.
The tensor $X$ is said to be \emph{left-orthogonal} if $n_1 \le n_2n_3$ and $(X^\mathrm{L})^\top \in \Stiefel(n_1, n_2n_3)$, and \emph{right-orthogonal} if $n_3 \le n_1n_2$ and $X^\mathrm{R} \in \Stiefel(n_3, n_1n_2)$.

As a TTD is not unique, certain orthogonality conditions can be enforced, which can improve the numerical stability of algorithms working with TTDs.
Those used in this work are given in \Cref{lemma:TangentConeLowRankVariety}.

\subsection{The Tangent Cone to the Low-Rank Variety}
\label{subsec:TangentConeLowRankVariety}
In \cite[Theorem~2.6]{kutschan2018}, a parametrization of the tangent cone to $\mathbb{R}_{\le (k_1, k_2)}^{n_1 \times n_2 \times n_3}$ is given and, because this parametrization is not unique, corresponding orthogonality conditions are added. The following lemma recalls this parametrization however with slightly different orthogonality conditions which make the proofs in the rest of the paper easier, and the numerical computations more stable because they enable to avoid matrix inversion in \Cref{alg:HOOI_proj}.

\begin{lemma}
\label{lemma:TangentConeLowRankVariety}
Let $X \in \mathbb{R}_{(r_1, r_2)}^{n_1 \times n_2 \times n_3}$ have $X = X_1 \cdot X_2'' \cdot X_3'' = X_1' \cdot X_2' \cdot X_3$ as TTDs, where $(X_2''^{\mathrm{L}})^\top \in \Stiefel(r_1, n_2 r_2)$, $X_3''^\top \in \Stiefel(r_2, n_3)$, $X_1' \in \Stiefel(r_1, n_1)$, and $X_2'^{\mathrm{R}} \in \Stiefel(r_2, r_1 n_2)$.
Then, $T_X \mathbb{R}_{\le (k_1, k_2)}^{n_1 \times n_2 \times n_3}$ is the set of all $G$ such that
\begin{equation}
\label{eq:paran_g_final}
G = \begin{bmatrix} X_1' && U_1 && W_1 \end{bmatrix} \cdot \begin{bmatrix}
X_2' && U_2 && W_2 \\ 
0 && Z_2 && V_2 \\
0 && 0 && X_2''
\end{bmatrix} \cdot \begin{bmatrix} W_3 \\ V_3 \\ X_3'' \end{bmatrix}
\end{equation}
with $U_1 \in \Stiefel(s_1, n_1)$, $W_1 \in \mathbb{R}^{n_1 \times r_1}$, $U_2 \in \mathbb{R}^{r_1 \times n_2 \times s_2}$, $W_2 \in \mathbb{R}^{r_1 \times n_2 \times r_2}$, $Z_2 \in \mathbb{R}^{s_1 \times n_2 \times s_2}$, $V_2 \in \mathbb{R}^{s_1 \times n_2 \times r_2}$, $W_3 \in \mathbb{R}^{r_2 \times n_3}$, $V_3^\top \in \Stiefel(s_2, n_3)$, $s_i = k_i-r_i$ for all $i \in \{1, 2\}$, and
\begin{equation}
\label{eq:orth_cond_K_modified}
\begin{aligned}
U_1^\top X_1' &= 0, & W_1^\top X_1' &= 0, & (U_2^\mathrm{R})^\top X_2'^{\mathrm{R}} &= 0,\\
W_3 X_3''^\top &= 0, & V_3 X_3''^\top &= 0, & V_2^\mathrm{L} (X_2''^{\mathrm{L}})^\top &= 0.
\end{aligned}
\end{equation}
\end{lemma}

\begin{proof}
{By \cite[Theorem 2.6]{kutschan2018}, $T_X\mathbb{R}_{\le (k_1, k_2)}^{n_1 \times n_2 \times n_3}$ is the set of all $G \in \mathbb{R}^{n_1 \times n_2 \times n_3}$ that can be decomposed as
\begin{equation*} 
\begin{split}
G &= \begin{bmatrix} X_1' && \dot{U}_1 && \dot{W}_1 \end{bmatrix} \cdot
\begin{bmatrix}
X_2' && \dot{U}_2 && \dot{W}_2 \\ 
0 && \dot{Z}_2 && \dot{V}_2 \\
0 && 0 && X_2' \\
\end{bmatrix} \cdot \begin{bmatrix} \dot{W}_3 \\ \dot{V}_3 \\ X_3 \end{bmatrix}, \\
\end{split}
\end{equation*}
with the orthogonality conditions
\begin{align}
\label{eq:orth_cond_K}
\dot{U}_1^\top X_1' &= 0, & \left( \dot{V}_2 \cdot X_3 \right)^{\mathrm{L}} \left(  \left( X_{2}' \cdot X_3 \right)^\mathrm{L} \right)^\top &= 0 , &\dot{V}_3 X_{3}^\top &= 0,\\
\dot{W}_1^\top  X_1'&= 0, & \left(\dot{W}_2^{\mathrm{R}} \right)^\top X_2'^{\mathrm{R}}&=0,& \left(\dot{U}_2^{\mathrm{R}} \right)^\top X_2'^{\mathrm{R}} &= 0 .  \nonumber
\end{align} 
The following invariances hold for all $B \in \mathbb{R}^{s_2 \times s_2}, C \in \mathbb{R}^{r_2 \times r_2}, Q \in \mathbb{R}^{s_1 \times s_1}$, and $R \in \mathbb{R}^{r_1 \times r_1}$:
\begin{equation*}
\begin{split}
G = &\begin{bmatrix} X_1' &~ \dot{U}_1 Q^{-1}  &~ \dot{W}_1 R^{-1} \end{bmatrix} \cdot \begin{bmatrix}
X_2'&~ \dot{U}_2 \cdot B &~ \dot{W}_2 \cdot C\\ 
0 &~ Q \cdot \dot{Z}_2\cdot B &~ Q \cdot \dot{V}_2 \cdot C \\
0 &~ 0 &~  R \cdot X_2'\cdot C \\
\end{bmatrix} \cdot \begin{bmatrix} \dot{W}_3 \\  B^{-1} \dot{V}_3 \\ C^{-1} X_3' \end{bmatrix}. \\
\end{split}
\end{equation*}
Then, if we define $U_1 := \dot{U}_1 Q^{-1}$, ${W}_1 := \dot{W}_1 R^{-1}$, ${U}_2 := \dot{U}_2 \cdot B$, $Z_2 := Q \cdot \dot{Z}_2\cdot B$, $V_2 := Q \cdot \dot{V}_2 \cdot C$, $X_2'' := R \cdot X_2'\cdot C$, $V_3 := B^{-1} \dot{V}_3$, and $X_3'' := C^{-1} X_3'$,
the matrices $B$, $C$, $Q$, and $R$ can be chosen such that $X_2''$ is left-orthogonal, $X_3''^\top \in \Stiefel(r_2, n_3)$, $V_3^\top \in \Stiefel(s_2, n_3)$, and $U_1 \in \Stiefel(s_1, n_1)$, e.g., using SVDs. 
Additionally, $\dot{W}_3$ can be decomposed as $\dot{W}_3 =  \dot{W}_3 X_3''^\top X_3'' +  {W}_3$. The two terms involving $\dot{W}_3$ and $\dot{W}_2 \cdot C$ can then be regrouped as
\begin{equation*} 
\begin{split}
&X_1' \cdot X_2' \cdot \dot{W}_3 + X_1' \cdot \dot{W}_2 \cdot CX_3''
= X_1' \cdot W_2 \cdot X_3''+  X_1' \cdot X_2' \cdot {W}_3,
\end{split}
\end{equation*}
where we have defined ${W}_2 := X_2' \cdot \dot{W}_3 X_3''^\top + \dot{W}_2 \cdot C$, obtaining the parametrization \eqref{eq:paran_g_final} satisfying \eqref{eq:orth_cond_K_modified}.} \qed
\end{proof}

Expanding \eqref{eq:paran_g_final} yields, by \eqref{eq:orth_cond_K_modified}, a sum of six mutually orthogonal TTDs:
\begin{equation}
\label{eq:g_expanded}
\begin{split}
G =&~ {W}_1 \cdot X_2''\cdot X_3'' + X_1' \cdot X_2' \cdot {W}_3 + X_1' \cdot {W}_2 \cdot X_3'' \\
&+ {U}_1 \cdot {V}_2 \cdot X_3'' + X_1' \cdot {U}_2 \cdot {V}_3 +  {U}_1 \cdot Z_2 \cdot {V}_3.
\end{split}
\end{equation}
Thus, the following holds:
\begin{align} \label{eq:param_g}
W_1 &= G^{\mathrm{L}} \left(\left(X_2'' \cdot X_3''\right)^{\mathrm{L}}\right)^\top, & W_2 &= X_1'^\top \cdot G \cdot X_3''^\top, & W_3 &= \left(\left(X_1'' \cdot X_2''\right)^{\mathrm{R}}\right)^\top G^{\mathrm{R}}, \nonumber \\
U_2 &= X_1'^\top \cdot G \cdot V_3^\top, & V_2 &= U_1^\top \cdot G \cdot X_3''^\top, & Z_2 &= U_1^\top \cdot G \cdot V_3^\top.
\end{align}
The first three terms in \eqref{eq:g_expanded} form the tangent space $T_X \mathbb{R}_{(r_1, r_2)}^{n_1 \times n_2 \times n_3}$, the projection onto which is described in \cite[Theorem~3.1 and Corollary~3.2]{Lubich_TT_time_int_2015}.

\section{The Proposed Approximate Projection}
\label{sec:ProposedApproximateProjection}
In this section, we prove \Cref{prop:U2_V2_Z2_orth_proj} and then use it to prove \Cref{thm:approx_proj_TC_TT}. Both results rely on the following observation. 
By \cite[Proposition~A.6]{LevinKileelBoumal2022}, since $T_X \mathbb{R}_{\le (k_1, k_2)}^{n_1 \times n_2 \times n_3}$ is a closed cone, for all $Y \in \mathbb{R}^{n_1 \times n_2 \times n_3}$ and $\hat{Y} \in \mathcal{P}_{T_X \mathbb{R}_{\le (k_1, k_2)}^{n_1 \times n_2 \times n_3}} Y$, it holds that $\langle Y-\hat{Y},\hat{Y}\rangle=0$ or, equivalently, $\langle Y,\hat{Y}\rangle = \lVert \hat{Y} \rVert^2$. Thus, all elements of $\mathcal{P}_{T_X \mathbb{R}_{\le (k_1, k_2)}^{n_1 \times n_2 \times n_3}} Y$ have the same norm and \eqref{eq:min_proj} can be rewritten as
\begin{equation}
\label{eq:proj_max_reform}
\mathcal{P}_{T_X \mathbb{R}_{\le (k_1, k_2)}^{n_1 \times n_2 \times n_3}} Y
= \argmax_{\substack{Z \in T_X \mathbb{R}_{\le (k_1, k_2)}^{n_1 \times n_2 \times n_3} \\ \langle Y, Z \rangle = \|Z\|^2}} \|Z\|
= \argmax_{\substack{Z \in {T_X \mathbb{R}_{\le (k_1, k_2)}^{n_1 \times n_2 \times n_3}} \\ \left\langle Y, Z \right\rangle = \|Z\|^2}} \left\langle Y, \frac{Z}{\|Z\|} \right\rangle.
\end{equation}

\begin{proposition}
\label{prop:U2_V2_Z2_orth_proj}
Let $X$ be as in \Cref{lemma:TangentConeLowRankVariety}. For every $Y \in \mathbb{R}^{n_1 \times n_2 \times n_3}$ and every $\hat{Y} \in \mathcal{P}_{T_X \mathbb{R}_{\le (k_1, k_2)}^{n_1 \times n_2 \times n_3}} Y$, if $U_1$ and $V_3$ are the parameters of $\hat{Y}$ in \eqref{eq:g_expanded}, then the parameters $W_1$, $W_2$, $W_3$, $U_2$, $V_2$, and $Z_2$ of $\hat{Y}$ can be written as
\begin{align}
\label{eq:V2_U2_Z2_proj_orth}
W_1 &= P_{X_1'}^\perp \left( Y \cdot X_3''^\top  \right)^\mathrm{L} \left( X_2''^{\mathrm{L}} \right)^\top, & W_3 &= X_2'^{\mathrm{R}} \left( X_1'^\top \cdot Y  \right)^{\mathrm{R}} P_{X_3''^\top}^\perp, \nonumber\\
U_2 &= \left[ P_{X_2'^{\mathrm{R}}}^\perp \left( X_1'^\top \cdot Y \right)^\mathrm{R} \right]^{r_1 \times n_2 \times n_3} \cdot V_3^\top, & W_2 &= X_1'^\top \cdot Y \cdot X_3''^\top, \\
V_2 &=  U_1^\top \cdot \left[ \left( Y \cdot X_3''^\top  \right)^\mathrm{L} P_{\left( X_2''^{\mathrm{L}} \right)^\top}^\perp\right]^{n_1 \times n_2 \times r_2}, & Z_2 &= U_1^\top \cdot Y \cdot V_3^\top \nonumber.
\end{align}
Furthermore, $Y_\parallel(U_1, V_3)$ defined as in \eqref{eq:g_expanded} with the parameters from~\eqref{eq:V2_U2_Z2_proj_orth} is a feasible point of \eqref{eq:proj_max_reform} for all $U_1$ and all $V_3$.
\end{proposition}

\begin{proof}
Straightforward computations show that $\langle Y, \hat{Y} \rangle = \langle Y_\parallel(U_1, V_3), \hat{Y} \rangle$ and $\big\langle Y_\parallel(U_1, V_3), Y-Y_\parallel (U_1, V_3) \big\rangle = 0$. Thus, $Y_\parallel(U_1, V_3)$ is a feasible point of \eqref{eq:proj_max_reform}. Since $\hat{Y}$ is a solution to \eqref{eq:proj_max_reform}, $\|Y_\parallel(U_1, V_3)\| \le \|\hat{Y}\|$. Therefore, if $\hat{Y} = 0$, then $Y_\parallel(U_1, V_3) = 0$ and consequently all parameters in \eqref{eq:V2_U2_Z2_proj_orth} are zero because of \eqref{eq:param_g}. Otherwise, by using the Cauchy--Schwarz inequality, we have
\begin{equation}
\label{eq:ProofPropU1V3}
\|\hat{Y}\|^2
= \langle Y, \hat{Y} \rangle
= \langle Y_\parallel(U_1, V_3), \hat{Y} \rangle
\le \|Y_\parallel(U_1, V_3)\| \|\hat{Y}\|
\le \|\hat{Y}\|^2,
\end{equation}
where the last inequality holds because $\hat{Y}$ is a solution to \eqref{eq:proj_max_reform}. It follows that the Cauchy--Schwarz inequality is an equality and hence there exists $\lambda \in (0,\infty)$ such that $Y_\parallel(U_1, V_3) = \lambda \hat{Y}$. By \eqref{eq:ProofPropU1V3}, $\lambda = 1$. Thus, because of \eqref{eq:param_g}, the parameters in \eqref{eq:V2_U2_Z2_proj_orth} are those of $\hat{Y}$.
\qed
\end{proof}

\begin{theorem}
\label{thm:approx_proj_TC_TT}
Let $X$ be as in \Cref{lemma:TangentConeLowRankVariety} with $(r_1, r_2) \ne (k_1, k_2)$.
The approximate projection that computes the parameters $U_1$ and $V_3$ of $\tilde{Y} \in \tilde{\mathcal{P}}_{T_X \mathbb{R}_{\le (k_1, k_2)}^{n_1 \times n_2 \times n_3}} Y$ in \eqref{eq:g_expanded} with \Cref{alg:HOOI_proj} and the parameters $W_1$, $W_2$, $W_3$, $U_2$, $V_2$, and $Z_2$ with \eqref{eq:V2_U2_Z2_proj_orth} satisfies \eqref{eq:ApproximateProjectionAngleCondition} with $\omega$ as in \eqref{eq:OurOmega} for all $\varepsilon$ and all $i_{\mathrm{max}}$ in \Cref{alg:HOOI_proj}.
\end{theorem} 

\begin{proof}
Let $(s_1, s_2) := (k_1-r_1, k_2-r_2)$ and $\hat{Y} \in \mathcal{P}_{T_X \mathbb{R}_{\le (k_1, k_2)}^{n_1 \times n_2 \times n_3}}Y$. Thus, $s_1 + s_2 > 0$.
Because $W_1$, $W_2$, $W_3$, $U_2$, $V_2$, and $Z_2$ are as in \eqref{eq:V2_U2_Z2_proj_orth}, it holds that $\tilde{Y} = Y_\parallel(U_1,V_3)$, thus $\tilde{Y}$ is a feasible point of \eqref{eq:proj_max_reform}, and hence \eqref{eq:ApproximateProjectionAngleCondition} is equivalent to $\|\tilde{Y}\| \ge \omega \|\hat{Y}\|$.
To compare the norm of $\tilde{Y}$ with the norm of $\hat{Y}$, $\hat{U}_1$ and $\hat{V}_3$ are defined as the parameters of $\hat{Y}$. 
From \eqref{eq:V2_U2_Z2_proj_orth}, and because all terms are mutually orthogonal, we have that
\begin{equation*}
\begin{split}
\big\lVert \tilde{Y} \big\rVert^2 =&~\Big\| \mathcal{P}_{T_X \mathbb{R}_{(r_1, r_2)}^{n_1 \times n_2 \times n_3}} Y \Big\|^2 + \Big\lVert P_{U_1} \big( Y \cdot X_3''^\top \big)^\mathrm{L} P_{\left( X_2''^{\mathrm{L}} \right)^\top}^\perp \left( X_3'' \otimes I_{n_2} \right) \Big\rVert^2 \\
&+\left\lVert \big( ( P_{U_1} \cdot Y  )^\mathrm{R} +  \left( I_{n_2} \otimes X_1'  \right)  P_{X_2'^{\mathrm{R}}}^\perp ( X_1'^\top \cdot Y )^\mathrm{R} \big) P_{V_3^{\top}} \right\rVert^2.
\end{split}
\end{equation*}
Now, assume that $s_2/(n_3-r_2) > s_1/(n_1-r_1)$ and consider the first iteration of \Cref{alg:HOOI_proj}. Because in the second step $V_3$ is obtained by a truncated SVD of $\big( ( P_{U_1} \cdot Y  )^\mathrm{R} +  \left( I_{n_2} \otimes X_1'  \right)  P_{X_2'^{\mathrm{R}}}^\perp ( X_1'^\top \cdot Y )^\mathrm{R} \big) P_{X_3''^{\top}}^\perp$ and, by using \eqref{eq:ineqs_svd_trunc2},
\begin{equation*} 
\begin{split}
\big\lVert \tilde{Y} \big\rVert^2 \ge&~\Big\lVert \mathcal{P}_{T_X \mathbb{R}_{(r_1, r_2)}^{n_1 \times n_2 \times n_3}} Y \Big\rVert^2 + \Big\lVert P_{U_1} \big( Y \cdot X_3''^\top \big)^\mathrm{L} P_{\left( X_2''^{\mathrm{L}} \right)^\top}^\perp \left( X_3'' \otimes I_{n_2} \right) \Big\rVert^2 \\
&+ \frac{s_2}{n_3-r_2} \left\lVert \left(( P_{U_1} \cdot Y  )^\mathrm{R} + \left( I_{n_2} \otimes X_1'  \right)  P_{X_2'^{\mathrm{R}}}^\perp ( X_1'^\top \cdot Y )^\mathrm{R} \right) P_{X_3''^\top}^\perp \right\rVert^2.\\
\end{split}
\end{equation*} 
Furthermore, because in the first step $U_1$ is obtained from the truncated SVD of $ P_{X_1'}^\perp \big( ( Y \cdot P_{X_3''^{\top}}^\perp )^\mathrm{L} +  ( Y \cdot X_3''^\top )^\mathrm{L} P_{ X_2''^{\mathrm{L}^\top} }^\perp \left( X_3'' \otimes I_{n_2} \right) \big)$ and by using \eqref{eq:ineqs_svd_trunc1},
\begin{equation*} 
\begin{split}
\big\lVert \tilde{Y} \big\rVert^2 \ge&~\Big\lVert \mathcal{P}_{T_X \mathbb{R}_{(r_1, r_2)}^{n_1 \times n_2 \times n_3}} Y \Big\rVert^2 + \frac{s_2}{n_3-r_2} \Big\lVert P_{\hat{U}_1} \big( Y \cdot X_3''^\top \big)^\mathrm{L} P_{\left( X_2''^{\mathrm{L}} \right)^\top}^\perp \left( X_3'' \otimes I_{n_2} \right) \Big\rVert^2 \\
&+ \frac{s_2}{n_3-r_2} \left\lVert \left(( P_{\hat{U}_1} \cdot Y  )^\mathrm{R} + \left( I_{n_2} \otimes X_1'  \right)  P_{X_2'^{\mathrm{R}}}^\perp ( X_1'^\top \cdot Y )^\mathrm{R} \right) P_{X_3''^\top}^\perp \right\rVert^2,\\
\end{split}
\end{equation*} 
where we have used that a multiplication with $\frac{s_2}{n_3-r_2}$ can only decrease the norm. The same is true for a multiplication with $ P_{\hat{V}_3^\top}$ and thus 
\begin{equation*} 
\begin{split}
&\big\lVert \tilde{Y} \big\rVert^2 \ge~ 
\frac{s_2}{n_3-r_2} \Big( \big\lVert \mathcal{P}_{T_X \mathbb{R}_{(r_1, r_2)}^{n_1 \times n_2 \times n_3}} Y \big\rVert^2 +  \big\lVert P_{\hat{U}_1} \big( Y \cdot X_3''^\top \big)^\mathrm{L} P_{\left( X_2''^{\mathrm{L}} \right)^\top}^\perp \left( X_3'' \otimes I_{n_2} \right) \big\rVert^2  \\
&+ \left\lVert \left(( P_{\hat{U}_1} \cdot Y  )^\mathrm{R} + \left( I_{n_2} \otimes X_1'  \right)  P_{X_2'^{\mathrm{R}}}^\perp ( X_1'^\top \cdot Y )^\mathrm{R} \right) P_{\hat{V}_3^\top} \right\rVert^2 \Big) =\frac{s_2}{n_3-r_2} \big\lVert \hat{Y} \big\rVert^2.\\
\end{split}
\end{equation*} 
In \Cref{alg:HOOI_proj}, the norm of the approximate projection increases monotonously. Thus, this lower bound is satisfied for any $\varepsilon$ and $i_{\max}$. 
A similar derivation can be made if $s_2/(n_3-r_2) \le s_1/(n_1-r_1)$.
\qed
\end{proof}

This section ends with three remarks on \Cref{alg:HOOI_proj}. First, the instruction ``$[U, S, V] \gets \mathrm{SVD}_s(A)$'' means that $USV^\top$ is a truncated SVD of rank $s$ of $A$. Since those SVDs are not necessarily unique, \Cref{alg:HOOI_proj} can output several $(U_1, V_3)$ for a given input, and hence the approximate projection is set-valued.

Second, the most computationally expensive operation in \Cref{alg:HOOI_proj} is the truncated SVD. The first step of the first iteration requires to compute either $s_1$ singular vectors of a matrix of size $n_1 \times n_2 n_3$ or $s_2$ singular vectors of a matrix of size $n_1 n_2 \times n_3$. 
All subsequent steps are computationally less expensive since each of them merely requires to compute either $s_1$ singular vectors of $A := P_{X_1'}^\perp \Big[ \big( Y \cdot V_3^\top \big)^\mathrm{L}, \big( Y \cdot X_3''^\top \big)^\mathrm{L} P_{ X_2''^{\mathrm{L}^\top} }^\perp \Big] \in \mathbb{R}^{n_1 \times n_2 (r_2+s_2)}$ or $s_2$ singular vectors of $B := \Big[ \big( U_1^\top \cdot Y  \big)^\mathrm{R}; P_{X_2'^{\mathrm{R}}}^\perp \big( X_1'^\top \cdot Y \big)^\mathrm{R} \Big] P_{X_3''^{\top}}^\perp \in \mathbb{R}^{(r_1+s_1) n_2 \times n_3}$, and, in general, $s_1+r_1 \ll n_1$ and $s_2+r_2 \ll n_3$. This is because
\begin{equation*}
\begin{split}
&P_{X_1'}^\perp \Big( \big( Y \cdot P_{V_3^\top} \big)^\mathrm{L} +  \big( Y \cdot X_3''^\top \big)^\mathrm{L} P_{ X_2''^{\mathrm{L}^\top} }^\perp \left( X_3'' \otimes I_{n_2} \right) \Big) = A \begin{bmatrix}
V_3 \otimes I_{n_2} \\ 
X_3'' \otimes I_{n_2}
\end{bmatrix},
\end{split}
\end{equation*} 
and the rightmost matrix, being in $\Stiefel(n_2(s_2+r_2),n_2 n_3)$ by \Cref{lemma:TangentConeLowRankVariety}, does not change the left singular vectors (and values). The argument for $B$ is similar.
The {\Matlab} implementation of \Cref{alg:HOOI_proj} that is used to perform the numerical experiment in \Cref{sec:NumericalExperiment} computes a subset of singular vectors (and singular values) using the \mlin{svd} function with \mlin{`econ'} flag.

Third, studying the numerical stability of \Cref{alg:HOOI_proj} would require a detailed error analysis, which is out of the scope of the paper. Nevertheless, the modified orthogonality conditions improve the stability compared to the approximate projection described in \cite[\S 5.4.4]{kutschan2019} because \Cref{alg:HOOI_proj} uses only orthogonal matrices to project onto vector spaces (it uses no Moore--Penrose inverse).

\begin{algorithm}
\setstretch{1.2}
\caption{Iterative method to obtain $U_1$ and $V_3$ of $\tilde{\mathcal{P}}_{T_X \mathbb{R}_{\le (k_1, k_2)}^{n_1 \times n_2 \times n_3}}Y$}
\label{alg:HOOI_proj}
\begin{algorithmic} [0]
\REQUIRE
$Y \in \mathbb{R}^{n_1 \times n_2 \times n_3}$, $X = X_1' \cdot X_2' \cdot X_3 = X_1 \cdot X_2'' \cdot X_3'' \in \mathbb{R}_{(r_1, r_2)}^{n_1 \times n_2 \times n_3}$, $\varepsilon > 0$, $i_{\mathrm{max}}, s_1, s_2 \in \mathbb{N} \setminus \{0\}$
\STATE \textbf{Initialize:} $i\gets0$, $V_3 \gets P_{X_3''^\top}^\perp$, $U_1 \gets P_{X_1'}^\perp$, $\eta_1 \gets 0$, $\eta_{\mathrm{new}}\gets\infty$
\IF{$s_2/(n_3-r_2) > s_1/(n_1-r_1)$}
\WHILE {$i<i_{\mathrm{max}}$ \AND $ | \eta_{\mathrm{new}} - \eta_1 | \leq \varepsilon$}
\STATE $\eta_1 \gets \eta_{\mathrm{new}}$, $i \gets i+1$
\STATE$\left[U_1, \sim, \sim\right] \gets \mathrm{SVD}_{s_1} \Big( P_{X_1'}^\perp \Big( \big( Y \cdot P_{V_3^\top} \big)^\mathrm{L} +  \big( Y \cdot X_3''^\top \big)^\mathrm{L} P_{ X_2''^{\mathrm{L}^\top} }^\perp \left( X_3'' \otimes I_{n_2} \right) \Big) \Big)$
\STATE$\left[\sim, S, V_3^\top \right] \gets \mathrm{SVD}_{s_2} \Big( \Big( \big( P_{U_1} \cdot Y  \big)^\mathrm{R} +  \left( I_{n_2} \otimes X_1'  \right)  P_{X_2'^{\mathrm{R}}}^\perp \big( X_1'^\top \cdot Y \big)^\mathrm{R} \Big) P_{X_3''^{\top}}^\perp \Big)$
\STATE
$\eta_{\mathrm{new}} \gets \lVert S \rVert^2 +\lVert P_{U_1} \big( Y \cdot X_3''^\top \big)^\mathrm{L} P_{ X_2''^{\mathrm{L}^\top} }^\perp \left( X_3'' \otimes I_{n_2} \right) \rVert^2$
\ENDWHILE
\ELSE
\WHILE {$i<i_{\mathrm{max}}$ \AND $ | \eta_{\mathrm{new}} - \eta_1 | \leq \varepsilon$}
\STATE $\eta_1 \gets \eta_{\mathrm{new}}$, $i \gets i+1$
\STATE$\left[\sim, \sim, V_3^\top \right] \gets \mathrm{SVD}_{s_2} \Big( \Big( \big( P_{U_1} \cdot Y  \big)^\mathrm{R} +  \left( I_{n_2} \otimes X_1'  \right)  P_{X_2'^{\mathrm{R}}}^\perp \big( X_1'^\top \cdot Y \big)^\mathrm{R} \Big) P_{X_3''^{\top}}^\perp \Big)$
\STATE$\left[U_1, S, \sim\right] \gets \mathrm{SVD}_{s_1} \Big( P_{X_1'}^\perp \Big( \big( Y \cdot P_{V_3^\top} \big)^\mathrm{L} +  \big( Y \cdot X_3''^\top \big)^\mathrm{L} P_{ X_2''^{\mathrm{L}^\top} }^\perp \left( X_3'' \otimes I_{n_2} \right) \Big) \Big)$
\STATE
$\eta_{\mathrm{new}} \gets \lVert S \rVert^2 +\lVert \left( I_{n_2} \otimes X_1'  \right)  P_{X_2'^{\mathrm{R}}}^\perp \big( X_1'^\top \cdot Y \big)^\mathrm{R} P_{V_3^\top} \rVert^2$
\ENDWHILE
\ENDIF
\ENSURE
$U_1, V_3$.
\end{algorithmic}
\end{algorithm}

\section{A Numerical Experiment}
\label{sec:NumericalExperiment}
To compute gradient-related directions in the tangent cone, the input tensor for \Cref{alg:HOOI_proj} would be the gradient of the continuously differentiable function that is considered. Since, in general, such tensors are dense, we consider in this section randomly generated pairs of dense tensors $(X, Y)$ with $X \in \mathbb{R}_{\le (k_1, k_2)}^{n_1 \times n_2 \times n_3}$ and $Y \in \mathbb{R}^{n_1 \times n_2 \times n_3}$, and compare the values of $\Big\langle \frac{\tilde{Y}}{\|\tilde{Y}\|}, \frac{Y}{\|Y\|} \Big\rangle$ obtained by computing the approximate projection $\tilde{Y}$ of $Y$ onto $T_X \mathbb{R}_{\le (k_1, k_2)}^{n_1 \times n_2 \times n_3}$ using \Cref{thm:approx_proj_TC_TT}, the tensor diagrams from \cite[\S 5.4.4]{kutschan2019}, which we have implemented in \Matlab, and the point output by the built-in \Matlab {} function \mlin{fmincon} applied to \eqref{eq:min_proj}. The latter can be considered as a benchmark for the exact projection. Since $\|Y\| \ge \|\mathcal{P}_{T_X \mathbb{R}_{\le (k_1, k_2)}^{n_1 \times n_2 \times n_3}} Y\|$, \eqref{eq:ApproximateProjectionAngleCondition} is satisfied if $\Big\langle \frac{\tilde{Y}}{\|\tilde{Y}\|}, \frac{Y}{\|Y\|} \Big\rangle \ge \omega$.

For this experiment, we set $(k_1, k_2) := (3, 3)$ and generate fifty random pairs $(X, Y)$, where $X \in \mathbb{R}^{5 \times 5 \times 5}_{(2,2)}$ and $Y \in \mathbb{R}^{5 \times 5 \times 5}$, using the built-in \Matlab~ function \mlin{randn}. For such pairs, the $\omega$ from \eqref{eq:OurOmega} equals $\frac{1}{3}$. We use \Cref{alg:HOOI_proj} with $\varepsilon := 10^{-16}$, which implies that $i_{\max}$ is used as stopping criterion.
In the left subfigure of \Cref{fig:ni5_ri2_si1_box_plots_comparison_50exp_epsilon}, the box plots for this experiment are shown for two values of $i_{\max}$. As can be seen, for both values of $i_{\max}$, the values of $\Big\langle \frac{\tilde{Y}}{\|\tilde{Y}\|}, \frac{Y}{\|Y\|} \Big\rangle$ obtained by the proposed approximate projection are close to those obtained by \mlin{fmincon} and are larger than those obtained by the approximate projection from \cite[\S 5.4]{kutschan2019}. We observe that $\Big\langle \frac{\tilde{Y}}{\|\tilde{Y}\|}, \frac{Y}{\|Y\|} \Big\rangle$ is always larger than $\frac{1}{3}$, which suggests that \eqref{eq:OurOmega} is a pessimistic estimate.
The middle subfigure compares ten of the fifty pairs. For one of these pairs, the proposed method obtains a better result than \mlin{fmincon}. This is possible since the \mlin{fmincon} solver does not necessarily output a global solution because of the nonconvexity of \eqref{eq:min_proj}. An advantage of the proposed approximate projection is that it requires less computation time than the \mlin{fmincon} solver (a fraction of a second for the former and up to ten seconds for the latter).
In the rightmost subfigure, the evolution of $\eta_{\mathrm{new}} - \eta_1$ is shown for one of the fifty pairs.
This experiment was run on a laptop with a AMD Ryzen 7 PRO 3700U processor (4 cores, 8 threads) having 13.7 GiB of RAM under Kubuntu 20.04. The {\Matlab} version is R2020a. The code is publicly available.\footnote{URL: \url{https://github.com/golikier/ApproxProjTangentConeTTVariety}}

\begin{figure}
\centering
\includegraphics[width= \linewidth]{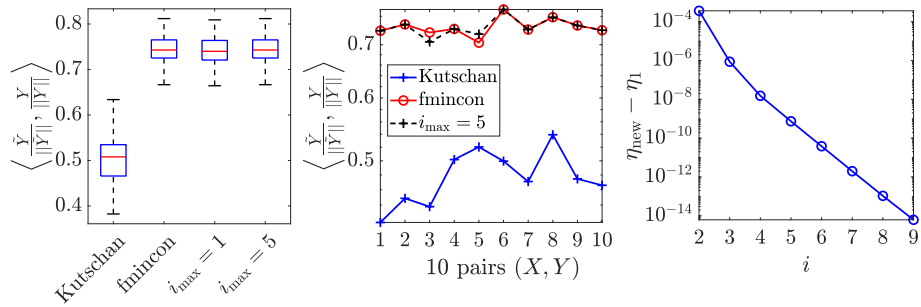}
\caption{A comparison of $\Big\langle \frac{\tilde{Y}}{\|\tilde{Y}\|}, \frac{Y}{\|Y\|} \Big\rangle$ for fifty randomly generated pairs $(X, Y)$, with $X \in \mathbb{R}^{5 \times 5 \times 5}_{(2,2)}$, $Y \in \mathbb{R}^{5 \times 5 \times 5}$, and $(k_1,k_2):=(3,3)$, for the approximate projection defined in \Cref{thm:approx_proj_TC_TT}, the one from \cite[\S 5.4]{kutschan2019}, and the one output by \mlin{fmincon}. On the rightmost figure, the evolution of $\eta_{\mathrm{new}} - \eta_1$ is shown for one of the fifty pairs.}
\label{fig:ni5_ri2_si1_box_plots_comparison_50exp_epsilon}
\end{figure}
%
%
%
\bibliographystyle{splncs04}
\bibliography{references}

\end{document}